\documentclass[12pt]{article}
\usepackage{graphicx}
\usepackage{amsmath}
\usepackage{amsfonts}
\usepackage{amssymb}
\usepackage{color}
\usepackage{tikz}
\providecommand{\U}[1]{\protect\rule{.1in}{.1in}}
\newtheorem{theorem}{Theorem}

\newtheorem{claim}[theorem]{Claim}

\newtheorem{lemma}[theorem]{Lemma}

\newtheorem{remark}[theorem]{Remark}

\newenvironment{proof}[1][Proof]{\noindent\textbf{#1.} }{\ \rule{0.5em}{0.5em}}
\begin{document}

\title{\textbf{The miracle of integer eigenvalues}\\
\hfill {\it\normalsize To the memory of Professor A. M. Vershik}}

\author{Richard Kenyon$^{1}$, Maxim Kontsevich$^{2}$, Oleg Ogievetsky$^{3,4,5}$,
\and Cosmin Pohoata$^{9}$, Will Sawin$^{6}$ and Senya Shlosman$^{3,5,7,8}$\\
\\
$^{1}$ Yale University, USA, $^{2}$ IHES, France\\$^{3}$Aix Marseille Univ, Universit\'e de Toulon, \\CNRS, CPT, Marseille, France\\$^{4}$Lebedev Physical Institute, Moscow, Russia,\\$^{5}$Inst. of the Information Transmission Problems, Moscow, Russia,\\$^{6}$Princeton University, USA, $^{7}$BIMSA, Beijing, China,\\$^{8}$Skolkovo Institute of Science and Technology, Moscow, Russia\\$^{9}$Emory University, USA}
\maketitle

\begin{abstract}
For partially ordered sets $X$ we consider the square matrices $M^{X}$ with
rows and columns indexed by linear extensions of the partial order on $X.$
Each entry $\left(  M^{X}\right)  _{PQ}$ is a formal variable defined by a
\textit{pedestal }of the linear order $Q$ with respect to linear order $P.$ We
show that all the eigenvalues of any such matrix $M^{X}$ are $\mathbb{Z}$ -linear
combinations of those variables.

\end{abstract}

\section{The statement of the main result}

Let $X=\left\{  \alpha_{1},...,\alpha_{n}\right\}  $ be a partially ordered
set with the partial order $\preccurlyeq.$ A linear extension $P$ of
$\preccurlyeq$ is a bijection $P:X\rightarrow\left[  1,...,n\right]  $, such
that for any pair $\alpha_{i},\alpha_{j},$ satisfying $\alpha_{i}
\preccurlyeq\alpha_{j}$ we have $P\left(  \alpha_{i}\right)  \leq P\left(
\alpha_{j}\right)  .$

\vskip.2cm Let $P,Q$ be two linear extensions of $\preccurlyeq.$ We call the
node $Q^{-1}\left(  k\right)  \in X$ a $\left(  P,Q\right)  $-disagreement
node (or descent node, following \cite{St}) iff $P\left(  Q^{-1}\left(  k-1\right)  \right)  >P\left(  Q^{-1}\left(
k\right)  \right)  .$
By
definition, the node $Q^{-1}\left(  1\right)  $ is a $\left(  P,Q\right)
$-agreement node. With every pair $P,Q$ we associate the function
$\varepsilon_{PQ}:\left\{  1,...,n-1\right\}  \rightarrow\left\{  0,1\right\}
,$ given by
\begin{equation}
\varepsilon_{PQ}\left(  k\right)  =\left\{
\begin{tabular}
[c]{ll}
$1$ & if $Q^{-1}\left(  k+1\right)  $ is a $\left(  P,Q\right)  $-descent
node,\\
$0$ & otherwise.
\end{tabular}
\ \ \ \ \right.  \label{21}
\end{equation}

Note that for some pairs $\left(  P,Q\right)  \neq$ $\left(  P^{\prime
},Q^{\prime}\right)  $ the functions $\varepsilon_{PQ},$ $\varepsilon
_{P^{\prime}Q^{\prime}}$ can coincide (see the Examples section).

\vskip.2cm To formulate our main result we denote by $\mathcal{E}=\left\{
\varepsilon:\left\{  1,...,n-1\!\right\}  \rightarrow\left\{  0,1\right\}
\!\right\}  $ the set of all $2^{n-1}$ different $\varepsilon$ functions, and we
associate with every $\varepsilon$ a corresponding formal variable
$a_{\varepsilon}.$ For any poset $X$ consider the square matrix $M^{X},$ whose
matrix elements are indexed by the pairs $\left(  P,Q\right)  ,$ and are given
by $\left(  M^{X}\right)  _{PQ}=a_{\varepsilon_{PQ}}.$

\vskip.2cm For example, the poset $\left(  X,\preccurlyeq\right)  $ with three elements
and one relation: $X=\{\{u,v,w\},u<v\}$ has three linear extensions of
$\preccurlyeq$: $u<v<w,\ u<w<v,\ w<u<v.$ Let $P$ be the linear extension
$u<v<w$ and $Q$ -- the linear extension $u<w<v$. We have $\varepsilon
_{PQ}=(0,1)$ since $2$ is not a descent ($u<v$ in both $Q$ and $P$) and
$3$ is a descent ($w<v$ in $Q$ but not in $P$). The matrix $M^{X}$
is
\begin{equation}\label{par31}
\left(
\begin{array}
[c]{ccc}
a_{00} & a_{01} & a_{10}\\
a_{01} & a_{00} & a_{10}\\
a_{01} & a_{10} & a_{00}
\end{array}
\right)  .
\end{equation}

The eigenvalues of this matrix are
$a_{00}-a_{01}$, $a_{00}-a_{10}$
and $a_{00}+a_{01}+a_{10}$, so they are $\mathbb{Z}$-linear combinations of the letters entering the
matrix. One of us (O.O.) conjectured that this holds (the eigenvalues are $\mathbb{Z}$-linear combinations of
the letters entering the matrix $M^X$)
for every poset $X$. Below we present the proof of this conjecture.

\begin{theorem}
For every poset $X$ the matrix $M^{X}$ is non-degenerate, and all its
eigenvalues are linear combinations of the variables $a_{\varepsilon}$ with
integer coefficients.
\end{theorem}

\noindent Here `non-degenerate' means non-degenerate over the field of rational functions in the matrix elements.

\vskip.2cm
The matrices $M^{X}$ were introduced in the paper \cite{OS}. It is proven
there that the row sums $\sum_{Q}\left(  M^{X}\right)  _{PQ}$ do not depend on
the row $P,$ so the matrix $M^{X}$ is `stochastic' (up to a scale), and $\Pi
_{X}\left(  \left\{  a_{\varepsilon}\right\}  \right) : =\sum_{Q}\left(
M^{X}\right)  _{PQ}$ is its main eigenvalue. In \cite{OS} the corresponding
sums are called the `pedestal polynomials'. They enter into the expression for
the generating functions of the monotone functions $f:X\rightarrow\left\{
0,1,2,...\right\}  $ (e.g. the generating function of the number of plane
partitions, spacial partitions, etc.):\newline
\begin{equation}
\sum_{\text{monotone }f:X\rightarrow\left\{  0,1,2,...\right\}  }t^{\sum_{x\in
X}f\left(  x\right)  }=\Pi_{X}\left(  t\right)  \prod_{k=1}^{n}\frac
{1}{1-t^{k}}, \label{23}
\end{equation}
where the polynomial $\Pi_{X}\left(  t\right)  $ is obtained from $\Pi
_{X}\left(  \left\{  a_{\varepsilon}\right\}  \right)  $ by the substitution
\[
a_{\varepsilon}\rightsquigarrow t^{\sum_{k=1}^{n-1}k\varepsilon\left(
k\right)  }.
\]
We put into the Appendix the relevant combinatorial facts about the pedestals
and pedestal polynomials.

\vskip.2cm
Our main tool is the filter semigroup of operators $M_{F}^{X},$ introduced in
the next section. They have appeared first in \cite{BHR, BD}, where their
spectral properties were studied. In fact, part of the proof of Theorem 1 can be obtained by following the proof of Thm
1.2 in \cite{BHR}. We give a shorter and more direct proof.

\vskip.2cm
The next section contains some general facts about posets. It is followed by
the section containing proofs.

\section{The filter semigroup}

At the end of this section we will introduce the filter semigroup. But it is
easier to describe it geometrically, as the face semigroup of a hyperplane
arrangement, so we do this first.

\subsection{Faces}

Consider the central real hyperplane arrangement $A_{n}$ consisting of
hyperplanes $\{H_{ij}:1\leq i<j\leq n\}$ in $\mathbb{R}^{n}$ defined by
$H_{ij}=\left\{  (x_{1},...,x_{n}):x_{i}=x_{j}\right\}  .$ Every open
connected component of the complement $\mathbb{R}^{n}\setminus\left\{  \cup
H_{ij}\right\}  $ is called a chamber. A \textit{cone} is any union of
closures of chambers which is \textit{convex}. Let us introduce the (finite)
set $\mathfrak{O}\left(  n\right)  $ of all different cones thus obtained.

\vskip.2cm
Let a poset $X$ of $n$ elements be given, with a binary relation
$\preccurlyeq$. To every pair $i,j\in X$ which is in the relation
$i\preccurlyeq j$ there corresponds a half-space $K_{ij}=\left\{  x_{i}\leq
x_{j}\right\}  \subset\mathbb{R}^{n}$ (here we assume that $X$ is identified
with $\{1,2,\dots,n\}$ as a plain set, ignoring the order). Consider the cone
\[
A\left(  X,\preccurlyeq\right)  =\left\{
{\displaystyle\bigcap\limits_{i,j:i\preccurlyeq j}}
K_{ij}\right\}  \in\mathfrak{O}\left(  n\right)
\]
where the intersection is taken over all pairs $i,j$ such that $i\preccurlyeq
j.$

\vskip.2cm
The following statements are well-known (and easy to prove), see \cite{B, D,
Sa, St}.

\begin{claim}
The above defined correspondence $\left(  X,\preccurlyeq\right)  \rightarrow
A\left(  X,\preccurlyeq\right)  $ is a one-to-one correspondence between the
set of all partial orders on $\left\{  1,2,...,n\right\}  $ and the set of all
cones $\mathfrak{O}\left(  n\right)  $.
\end{claim}

We present an illustration of this claim for $n=4$.
$$\vspace{-1.4cm}$$

\begin{figure}[th]
\vspace{-.0cm} \centering
\includegraphics[scale=0.4]{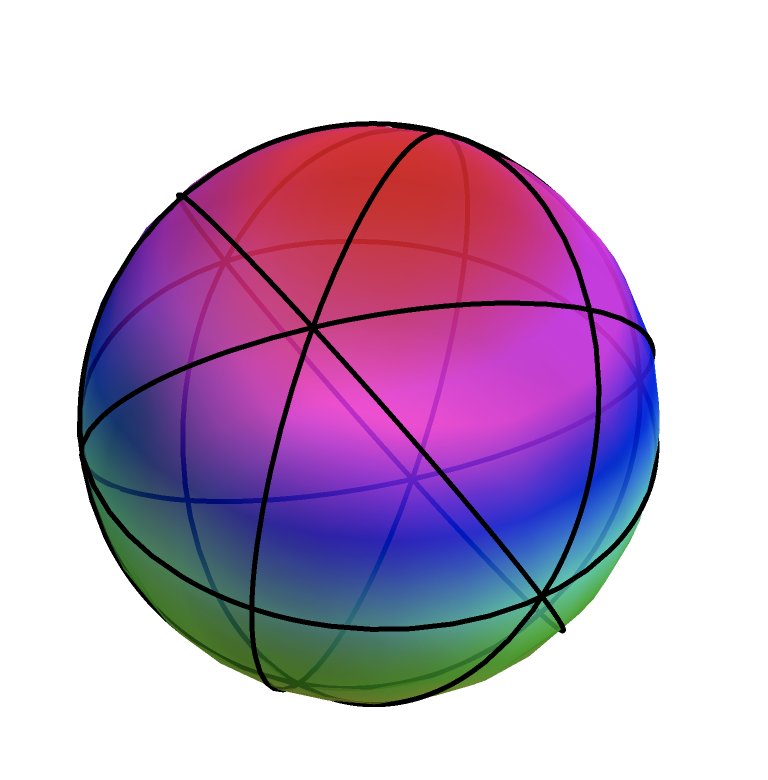}\caption{\footnotesize The central real hyperplane arrangement $A_{4}$ in $\mathbb{R}^{4},$ projected
to $\mathbb{R}^{3}$ along the line $x=y=z=t$ and intersected with the sphere
$\mathbb{S}^{2}\subset\mathbb{R}^{3}.$ It is a partition of $\mathbb{S}^{2}$
into 24 equal triangles, each with the angles $\left(  \frac{\pi}{2},\frac
{\pi}{3},\frac{\pi}{3}\right)  .$ The types of convex unions of the triangles
are: the sphere, the hemisphere, the region between two great semicircles, an
elementary triangle -- or e-triangle, a pair of e-triangles with a common side,
a triangle made from three e-triangles, a `square' formed by four e-triangles with
a common $\frac{\pi}{2}$-vertex, a triangle made from a `square' and a fifth
adjacent e-triange, a triangle formed by six e-triangles with a common
$\frac{\pi}{3}$-vertex. The number of corresponding convex shapes are 1, 12,
60, 24, 36, 48, 6, 24, 8, with total being 219. This is precisely the number
of partial orders on the set of four distinct elements, see the sequence A001035
in OEIS \cite{Sl}.}
\label{}
\end{figure}

Let $f^{\prime},f^{\prime\prime}$ be two faces in $A(X)=A\left(
X,\preccurlyeq\right)  .$ (It is allowed that one or both of them are in fact
chambers, i.e. faces of highest dimension). Define the face $f=f^{\prime
\prime}\left(  f^{\prime}\right)  \in A\left(  X\right)  $ -- or the
\textit{face-product} $f^{\prime\prime}f^{\prime}$ -- by the following
procedure: choose points $x^{\prime}\in f^{\prime},$ $x^{\prime\prime}\in
f^{\prime\prime}$ in general position and let $s_{x^{\prime}x^{\prime\prime}
}:\left[  0,1\right]  \rightarrow\mathbb{R}^{n}$ be a linear segment,
$s_{x^{\prime}x^{\prime\prime}}\left(  0\right)  =x^{\prime},$ $s_{x^{\prime
}x^{\prime\prime}}\left(  1\right)  =x^{\prime\prime}.$ Consider the face
$f\in A\left(  X\right)  $ which contains all the points $s_{x^{\prime
}x^{\prime\prime}}\left(  1-\varepsilon\right)  $ of our segment for
$\varepsilon>0$ small enough. Such a face does exist due to the convexity of
$A(X).$ By definition, $f^{\prime\prime}\left(  f^{\prime}\right)  =f.$ Note
that if $f^{\prime\prime}$ is a chamber then $f^{\prime\prime}f^{\prime
}=f^{\prime\prime}.$

\vskip.2cm
The face-product is associative. We mention for completeness that the
semigroups $A\left(  X,\preccurlyeq\right)  $ are what are called
\textit{left-regular bands, }see \cite{Sa}:

\begin{claim}
For every choice of faces $f,g,h\in A\left(  X,\preccurlyeq\right)  $ we have
\[
f\left(  gh\right)  =\left(  fg\right)  h,
\]
\[
ff=f,\ \ fgf=fg.
\]

\end{claim}

We do not give here the proofs as we are not using these relations.

\subsection{Filters}

Let $F$ be a filter on $X$ of rank $k,$ i.e. a surjective map $F:X\rightarrow
\left\{  1,...,k\right\}  ,$ preserving the partial order, and let

\[
\left\{  b_{1},...,b_{j_{1}}\right\}  ,\left\{  b_{j_{1}+1},...,b_{j_{2}
}\right\}  ,...,\left\{  b_{j_{k-1}+1},...,b_{j_{k}}\right\}  \subset X
\]

\vskip .1cm
\noindent be its `floors':\newline
\[
\left\{  b_{j_{r-1}+1},...,b_{j_{r}}\right\}  =F^{-1}\left(  r\right)
,\ r=1,...,k.
\]

\vskip .1cm
\noindent Consider the face $f_{F}\in A\left(  X,\preccurlyeq\right)  ,$ defined by the
equations

\[
x_{b_{j_{r-1}+1}}=...=x_{b_{j_{r}}},\ r=1,...,k
\]

\vskip .1cm
\noindent
and inequalities

\[
x_{b_{j_{1}}}<x_{b_{j_{2}}}<...<x_{b_{j_{k}}}.
\]

\vskip .1cm
\noindent (More precisely, we write an equation for every floor of $F$ which contains at
least two elements of $X.$) This is a one-to-one correspondence between faces
and filters. The filters of the highest rank $n,$ i.e. the linear extensions
of $\preccurlyeq,$ correspond to the chambers.

\vskip.2cm
The corresponding filter-product looks as follows. For $F^{\prime},$
$F^{\prime\prime}$ being two filters of $X,$ the filter $F=F^{\prime\prime
}F^{\prime}$ on $X$ is uniquely defined by the following properties:

\begin{itemize}
\item For $u,v$ with $F^{\prime\prime}\left(  u\right)  <F^{\prime\prime
}\left(  v\right)  \ $we have $F\left(  u\right)  <F\left(  v\right)  .$

\item For $u,v$ with $F^{\prime\prime}\left(  u\right)  =F^{\prime\prime
}\left(  v\right)  \ $we have $F\left(  u\right)  <F\left(  v\right)  $ iff
$F^{\prime}\left(  u\right)  <F^{\prime}\left(  v\right)  .$\newline
\end{itemize}

Indeed, let $f^{\prime},f^{\prime\prime}$ be the two faces, corresponding to
the filters $F^{\prime},F^{\prime\prime},$ and the general position points
$x^{\prime},x^{\prime\prime}$ belong to corresponding faces.

\vskip.2cm
The fact that $F^{\prime\prime}\left(  u\right)  <F^{\prime\prime}\left(
v\right)  $ means that $x_{u}^{\prime\prime}<x_{v}^{\prime\prime}.$ But the
point $s_{x^{\prime}x^{\prime\prime}}\left(  1-\varepsilon\right)  $ is close
to the point $x^{\prime\prime},$ therefore $\left[  s_{x^{\prime}
x^{\prime\prime}}\left(  1-\varepsilon\right)  \right]  _{u}<\left[
s_{x^{\prime}x^{\prime\prime}}\left(  1-\varepsilon\right)  \right]  _{v}$ for
all $\varepsilon$ small enough.

\vskip.2cm
The fact that $F^{\prime\prime}\left(  u\right)  =F^{\prime\prime}\left(
v\right)  \ $while $F^{\prime}\left(  u\right)  <F^{\prime}\left(  v\right)  $
means that $x_{u}^{\prime\prime}=x_{v}^{\prime\prime}$ while $x_{u}^{\prime
}<x_{v}^{\prime}.$ Since the map $s_{x^{\prime}x^{\prime\prime}}:\left[
0,1\right]  \rightarrow\mathbb{R}^{n}$ is linear, for any $t<1$ we have
$\left[  s_{x^{\prime}x^{\prime\prime}}\left(  t\right)  \right]  _{u}<\left[
s_{x^{\prime}x^{\prime\prime}}\left(  t\right)  \right]  _{v}.$\bigskip

Let $F$ be a filter on $X$, and $P$ is some filter of rank $n$, i.e. a linear
order on $X.$ Then the filter $FP$ is again a filter of rank $n.$ Consider the
square matrix $M_{F}^{X}=\left\Vert \left(  M_{F}^{X}\right)  _{P,Q}
\right\Vert $ where $P,Q$ are linear orders on $X:$
\[
\left(  M_{F}^{X}\right)  _{P,Q}=\left\{
\begin{tabular}
[c]{ll}
$1$ & if $Q=FP$\\
$0$ & if $Q\neq FP$
\end{tabular}
\ \ \right.  .
\]
The operators $M_{F}^{X}$ play a central role in our proof.\newline

Examples of the operators $M_{F}^{X}$ are given in the Examples section below.

\section{Proof of the main result}

The plan of the proof is the following:

\begin{enumerate}
\item We will show that the matrix $M^{X}$ can be written as a linear
combination of $M_{F}^{X}$-s with integer monomial coefficients.

\item We will show that all $M_{F}^{X}$-s can be made upper-triangular via
conjugation with the \textbf{same }matrix, and the resulting upper-triangular
matrices have integer entries on the diagonal.
\end{enumerate}

\subsection{The filter decomposition}

Let us rewrite $M^{X}$ as the sum over all $2^{n-1}$ functions $\varepsilon
:\left\{  1,...,n-1\right\}  \rightarrow\left\{  0,1\right\}  $:
\begin{equation}
M^{X}=\sum_{\varepsilon}a_{\varepsilon}B_{X,\varepsilon}, \label{bmu}
\end{equation}
where the entries of each matrix $B_{X,\varepsilon}$ are $0$ or $1.$

\vskip.2cm
For every function $\varepsilon$ we define the number $r\left(  \varepsilon
\right)  =1+\sum_{j=1}^{n-1}\varepsilon\left(  j\right)  ,$ and we partition
the segment $\left\{  1,...,n\right\}  $ into $r\left(  \varepsilon\right)  $
consecutive segments
\begin{align*}
\left\{  1,...,n\right\}   &  =\left\{  1,...,c_{1}\right\} \\
&  \cup\left\{  c_{1}+1,...,c_{1}+c_{2}\right\} \\
&  \cup\left\{  c_{1}+c_{2}+1,...,c_{1}+c_{2}+c_{3}\right\}  \cup...\\
&  \cup\left\{  c_{1}+...+c_{r\left(  \varepsilon\right)  }+1,...,n\right\}  ,
\end{align*}
where the values $c_{1}+1,c_{1}+c_{2}+1,...,c_{1}+...+c_{r\left(
\varepsilon\right)  }+1$ are all the points where the function $\varepsilon$
takes value $1.$

For $c_{1},...,c_{r}$ being integers summing
up to $n$ we denote by $\mathcal{F}_{c_{1},...,c_{r}}$ the set of all
filters $F:X\rightarrow\left[  1,2,...,r\right]  $ such that $\left\vert
F^{-1}\left(  i\right)  \right\vert =c_{i}$ for all $i=1,...,r.$

\begin{lemma}
Suppose that the matrix $B_{X,\varepsilon}\neq0,$ and the function
$\varepsilon$ has the parameters $r$ and $c_{1},...,c_{r}$. Then the following
inclusion-exclusion identity holds:
\begin{align}
B_{X,\varepsilon}  &  =\sum_{F\in\mathcal{F}_{c_{1},...,c_{r}}}M_{F}
^{X}-\left[  \sum_{\substack{F\in\mathcal{F}_{c_{1}+c_{2},c_{3},...,c_{r}}
\cup\\\cup\mathcal{F}_{c_{1},c_{2}+c_{3},...,c_{r}}\cup...}}M_{F}^{X}\right]
\label{15}\\
&  +\left[  \sum_{\substack{F\in\mathcal{F}_{c_{1}+c_{2}+c_{3},c_{4}
,...,c_{r}}\cup\\\cup\mathcal{F}_{c_{1}+c_{2},c_{3}+c_{4},...,c_{r}}\cup
...}}M_{F}^{X}\right]  -...\nonumber
\end{align}
where the sums are taken over all possible mergers of neighboring indices
$c_{i},$ and the signs are $\left(  -1\right)  ^{\#mergers}.$
\end{lemma}

\begin{proof}
Indeed, if we take an order $Q$ from the row $P$ which appears in the lhs,
then it agrees with $P$ over the first $c_{1}-1$ locations, then it disagrees
once, then it agrees again over next $c_{2}-1$ locations, then disagrees once
again, etc. But an order $Q$ from the row $P$ which appears in the rhs and
corresponds to the first sum in $\left(  \ref{15}\right)  ,$ agrees with $P$
over the first $c_{1}-1$ locations, then it \textbf{agrees or disagrees} once,
then it agrees again over next $c_{2}-1$ locations, then \textbf{agrees or
disagrees} once again, etc. Therefore we have to remove all these $Q$-s which
agrees with $P$ over the first $c_{1}-1$ locations, then \textbf{agrees }once
again, then agrees also over next $c_{2}-1$ locations, etc.

\vskip.2cm
See the Examples section for some $M_{F}^{X}$ operators.
\end{proof}

\subsection{Conjugation of $M_{F}^{X}$-s to upper-triangular}

Let $X=\left\{  \alpha_{1},...,\alpha_{n}\right\}  $ be a poset
with the partial order $\preccurlyeq.$  We denote by $\mathsf{{Tot}}_{X}$ the set of all
total orders extending $\preccurlyeq$.
Our matrices $M_{F}^{X}$ are of the size $\left\vert \mathsf{{Tot}}
_{X}\right\vert \times\left\vert \mathsf{{Tot}}_{X}\right\vert $. Let us now
abolish all order relations on $X,$ getting the poset $\bar{X}$ with
$\left\vert \mathsf{{Tot}}_{\bar{X}}\right\vert =n!\ .$ Of course, $M_{F}^{X}$
is a submatrix of $M_{F}^{\bar{X}}.$ Imagine (after reindexing) that it is an
upper-left submatrix. We claim that to the right of this submatrix all matrix
elements of $M_{F}^{\bar{X}}$ are zero, and so $M_{F}^{X}$ is a block of
$M_{F}^{\bar{X}}$. Indeed, each row of $M_{F}^{\bar{X}}$ has exactly one $1,$
and the rest are $0$-s. But each row of $M_{F}^{X}$ already has one $1.$ So it
is sufficient to know that the spectrum of $M_{F}^{\bar{X}}$ consists of integers.

\vskip.2cm In what follows, the initial poset $X$ will not appear any more,
and we will deal only with `totally unordered' poset $\bar{X}.$ The fact that
the matrices $M_{F}^{\bar{X}}$
can be conjugated simultaneously to upper-triangular ones
can be deduced from the
results of the papers \cite{BHR, BD}. We give a shorter and more direct proof.

\vskip.2cm Let us consider an even bigger matrix, $N_{F}^{\bar{X}}$, of size
$2^{n\left(  n-1\right)  /2}.$ Here $F$ is a filter on $X,$ while the rows and
columns of $N_{F}^{\bar{X}}$ are indexed by \textit{tournaments }between the
$n$ entries of $\bar{X}$. A \textit{tournament} is an assignment of an order
$\preccurlyeq$ to each pair $i\neq j$ of the elements of the set $\bar{X},$
independently for each pair. If we have a tournament $\preccurlyeq$ and a
filter $F$ on $\bar{X},$ then we define a new tournament $\preccurlyeq_{F}$ by
the rule:

\begin{enumerate}
\item If $F\left(  i\right)  =F\left(  j\right)  $ then $i\preccurlyeq_{F}j$
iff $i\preccurlyeq j,$

\item If $F\left(  i\right)  <F\left(  j\right)  $ then $i\preccurlyeq_{F}j$.
\end{enumerate}

We define $N_{F}^{\bar{X}}$ by
\[
\left(  N_{F}^{\bar{X}}\right)  _{\preccurlyeq\preccurlyeq^{\prime}}=\left\{
\begin{tabular}
[c]{ll}
$1$ & if $\preccurlyeq^{\prime}=\ \preccurlyeq_{F}$\\
$0$ & if $\preccurlyeq^{\prime}\not =\ \preccurlyeq_{F}$
\end{tabular}
\ \ \ \right.
\]

Any linear order defines a tournament in an obvious way, so our matrices
$M_{F}^{\bar{X}}$ are blocks of $N_{F}^{\bar{X}}$-s$,$ and it is sufficient to
study $N_{F}^{\bar{X}}$-s. \newline

The key observation now is the fact that $N_{F}^{\bar{X}}$ is a tensor product
of $n\left(  n-1\right)  /2$ two-by-two matrices, corresponding to all pairs
$\left(  i,j\right)  ,$ since the tournament orders $\preccurlyeq$ can be
assigned to the pairs independently. And since the tensor product of upper
triangular matrices is upper triangular, it is sufficient to check our claim
just for the filters and tournaments in the case $n=\left\vert \bar
{X}\right\vert =2.$

\vskip.2cm
The two-element no-order set $\bar{X}=\left\{  1,2\right\}  $ carries three
different filters and has two possible tournaments. The three two-by-two
matrices $N_{F}^{\bar{X}}$-s are $N_{1}:=\left(
\begin{array}
[c]{cc}
1 & 0\\
1 & 0
\end{array}
\right)  ,N_{2}:=\left(
\begin{array}
[c]{cc}
1 & 0\\
0 & 1
\end{array}
\right)  ,$ and $N_{3}:=\left(
\begin{array}
[c]{cc}
0 & 1\\
0 & 1
\end{array}
\right)  .$ \textbf{ }Conjugating them by the discrete Fourier transform
matrix $U=\frac{1}{\sqrt{2}}\left(
\begin{array}
[c]{cc}
1 & 1\\
1 & -1
\end{array}
\right)  $ brings them to the triple of upper triangular matrices:
$UN_{1}U^{-1}=\left(
\begin{array}
[c]{cc}
1 & 1\\
0 & 0
\end{array}
\right)  ,UN_{2}U^{-1}=\left(
\begin{array}
[c]{cc}
1 & 0\\
0 & 1
\end{array}
\right)  ,$ and $UN_{3}U^{-1}=\left(
\begin{array}
[c]{cc}
1 & -1\\
0 & 0
\end{array}
\right)  .$ Extending the conjugation through the tensor product finishes the proof.

\begin{remark}
Recall the definition (\ref{bmu}) of the matrices $B_{X,\varepsilon}$: for a poset $X$,
the set of (0,1)-valued matrices $\{B_{X,\varepsilon}\}_{\varepsilon
\in\{0,1\}^{\{1,\dots,n-1\}}}$ is defined by $M^{X}=\sum_{\varepsilon
}a_{\varepsilon}B_{X,\varepsilon}$. Let $\mathcal{L}(X)$ be the Lie algebra
generated by the matrices $\{B_{X,\varepsilon}\}$. The proof shows that the
Lie algebra $\mathcal{L}(X)$ is solvable.
\end{remark}

\begin{remark}
Let us denote by $\Phi_{T}$ the algebra of functions on the set $\mathsf{Tour}
_{\bar{X}}$ of tournaments considered as the set of vertices of the
$n(n-1)/2$-dimensional cube in $\mathbb{R}^{n(n-1)/2}$. This algebra carries
an increasing filtration by subspaces
\[
0\subset\Phi_{T}^{\leq0}\subset\Phi_{T}^{\leq1}\subset\dots\subset\Phi
_{T}^{\leq{\frac{n(n-1)}{2}}}=\Phi_{T}
\]
consisting of restrictions of polynomials of degree $\leq0,\leq1,\dots$ to the
vertices of the cube. This filtration is \textbf{strictly multiplicative} in
the sense that
\[
\Phi_{T}^{\leq k}=\underbrace{\Phi_{T}^{\leq1}\cdot\dots\cdot\Phi_{T}^{\leq1}
}_{k\text{ times}}\,.
\]
Our considerations imply that all operators $N_{F}^{\bar{X}}$ preserve this
filtration, and commute with each other on the associated graded space
$\oplus_{k}\Phi_{T}^{\leq k}/\Phi_{T}^{\leq k-1}$.

\vskip.2cm
Restricting functions from $\Phi_{T}$ to the subset $\mathsf{Tot}_{X}
\subset\mathsf{Tour}_{\bar{X}}$ we obtain again a strictly multiplicative
filtration on the algebra $\Phi_{X}:=\mathbb{R}^{\mathsf{Tot}_{X}}$ of
functions on $\mathsf{Tot}_{X}$, preserved by all operators $M_{F}^{X}$ where
$F$ runs through filters on the poset $X$.
\end{remark}

\section{Appendices}

\subsection{Pedestals}

\vskip.2cm Let again $X$ be a finite poset with the partial order
$\preccurlyeq,$ and $P,Q$ be a pair of linear orders on $X,$ consistent with
$\preccurlyeq.$ We define the function $q_{PQ}$ on $X$ by
\begin{equation}
q_{PQ}\left(  Q^{-1}\left(  k\right)  \right)  =\#\left\{  l:l\leq
k,Q^{-1}\left(  l\right)  \text{ is a }\left(  P,Q\right)  \text{-descent
node}\right\}  . \label{13}
\end{equation}
Clearly, the function $q_{PQ}$ is non-decreasing on $X,$ and $q_{PQ}\left(
Q^{-1}\left(  1\right)  \right)  =0.$ It is called the pedestal of $Q$ with
respect to $P.$

\vskip.2cm
For example, let $X$ be a $3\times2$ Young diagram, and
\[
P=\left[
\begin{tabular}
[c]{lll}
1 & 2 & 3\\
4 & 5 & 6
\end{tabular}
\ \right]  ,\ Q=\left[
\begin{tabular}
[c]{lll}
1 & 2 & 5\\
3 & 4 & 6
\end{tabular}
\ \right]
\]
be the two standard tableaux. Then
\[
q_{PQ}=\left[
\begin{tabular}
[c]{lll}
0 & 0 & 1\\
0 & 0 & 1
\end{tabular}
\ \right]  .
\]

Let $\mathcal{E}_{P}$ denotes the set of all pedestals $q_{PQ}.$ The correspondence
$$Q\to q_{PQ} \in \mathcal{E}_{P}$$ is a one-to-one map, as explained below.

Clearly,
there is a map $\mathcal{E}_{P}\rightarrow\mathcal{E},$ which to every
pedestal $q_{PQ}$ corresponds its `discrete derivative' $\varepsilon_{PQ}$.

\vskip.2cm The pedestals were introduced in \cite{S} in the following context.
Consider the set $\mathcal{P}=\mathcal{P}_{X}$ of all non-negative
integer-valued non-decreasing functions $p$ on $X.$ Denote by $v\left(
p\right)  $ the `volume' of $p:$
\[
v\left(  p\right)  =\sum_{\alpha\in X}p\left(  \alpha\right)  ,
\]
and let $G$ be the following generating function:
\[
G_{X}\left(  t\right)  =\sum_{k\geq0}g_{k}t^{k}=\sum_{p\in\mathcal{P}_{X}
}t^{v\left(  p\right)  },
\]
i.e. $g_{k}$ is the number of non-decreasing $p$-s with $v\left(  p\right)
=k.$ For example, if the poset $X$ is in fact the set $X_{n}=\left[
1,2,...,n\right]  ,$ ordered linearly, then
\[
G_{X_{n}}\left(  t\right)  =\prod_{l=1}^{n}\frac{1}{1-t^{l}}
\]
is the generating function of the sequence $g_{k}$ of the number of partitions
$\pi$ of the integer $k$ into at most $n$ parts: $k=\pi\left(  1\right)
+\pi\left(  2\right)  +...+\pi\left(  n\right)  ,$ with $\pi\left(  i\right)
\geq0,$ $\pi\left(  i\right)  \leq\pi\left(  i+1\right)  .$ Let $\mathcal{Y}
_{n}$ denote the set of all such partitions $\pi$ (i.e. Young diagrams).

\vskip.2cm
In order to write a formula for $G_{X}$ for an arbitrary poset $X$ one needs
pedestals. Namely, let us fix some ordering $P$ of $X,$ consider all pedestals
$q_{PQ},$ and let
\begin{equation}
\Pi_{P}\left(  t\right)  =\sum_{Q}t^{v\left(  q_{PQ}\right)  } \label{12}
\end{equation}
be the generating function (in fact, generating polynomial) of the sequence of
the number of pedestals with a given volume. Then we have the identity:
\begin{equation}
G_{X}\left(  t\right)  =\Pi_{P}\left(  t\right)  G_{X_{n}}\left(  t\right)
\equiv\Pi_{P}\left(  t\right)  \prod_{l=1}^{n}\frac{1}{1-t^{l}}, \label{11}
\end{equation}
(compare with $\left(  \ref{23}\right)  $). In particular, it follows from
$\left(  \ref{11}\right)  $ that the polynomial $\Pi_{P}\left(  t\right)  $
does not depend on $P,$ and thus can be denoted by $\Pi_{X}\left(  t\right)
.$ The reason for $\left(  \ref{11}\right)  $ to hold is the existence of the
bijection $b:\mathcal{P}_{X}\to \mathcal{E}_{P}\times\mathcal{Y}_{n}$
between the set $\mathcal{P}_{X}$ of nondecreasing functions and the
direct product $\mathcal{E}_{P}\times\mathcal{Y}_{n},$ respecting the volumes.
Namely, to each pedestal $q_{PQ}$ and each partition $\pi$ it associates the
following function $p$ on $X:$
\[
p\left(  Q^{-1}\left(  k\right)  \right)  =q_{PQ}\left(  Q^{-1}\left(
k\right)  \right)  +\pi\left(  k\right)  ,\ k=1,...,n.
\]
Clearly, the function thus defined is non-decreasing on $X.$ For the check
that $b$ is a one-to-one correspondence see \cite{S}, relation (46) and the construction of the inverse map $b^{-1}$ there.
The bijectivity of $b$ implies in particular that for each $P$ all the pedestals $q_{PQ}$ are distinct.

\vskip.2cm In the case when $X$ is a (2D) Young diagram, the functions
$p\in\mathcal{P}_{X}$ are called `reverse plane partitions'. The generating
function $G_{X}$ for these is also given by the famous Stanley  \cite{St} formula,
\[
G_{X}\left(  t\right)  =\prod_{\alpha\in X}\frac{1}{1-t^{h\left(
\alpha\right)  }},
\]
where $h\left(  \alpha\right)  $ is the hook length of the cell $\alpha\in X.$
When $X$ is a rectangle, this is the MacMahon
formula.
That means that for the case of $X$ being a Young diagram nice cancellations
happen in the rhs of $\left(  \ref{11}\right)  .$ One can check that for some
$X$ being a 3D Young diagram no cancellations happen in $\left(
\ref{11}\right)  ,$ and this is the reason why the analog of the Stanley
formula in the 3D case does not exist.

\subsection{Pedestal polynomials}

The fact that the function $\Pi_{P}\left(  t\right)  $ (see $\left(
\ref{12}\right)  $) does not depend on the order $P$ on $X,$ but only on $X,$
has the following generalization. Instead of characterizing the pedestal
$q_{PQ}$ just by its volume let us associate with it the monomial
$m_{PQ}\left(  x_{1},x_{2},x_{3},...\right)  =x_{1}^{l_{1}-1}x_{2}
^{l_{2}-l_{1}}...x_{r}^{l_{r}-l_{r-1}}x_{r+1}^{n-l_{r}+1},$ where $r$ is the
number of $\left(  P,Q\right)  $-descent nodes, and $l_{1},...,l_{r}$ are
their locations, see $\left(  \ref{13}\right)  .$ Note that $m_{PQ}\left(
1,t,t^{2},...\right)  =t^{v\left(  q_{PQ}\right)  }.$

\vskip.2cm It was shown in \cite{OS} that the polynomial
\[
\mathfrak{h}_{P}\left(  x_{1},x_{2},x_{3},...\right)  =\sum_{Q\in
\mathsf{{Tot}}_{X}}m_{PQ}\left(  x_{1},x_{2},x_{3},...\right)
\]
is also independent of $P,$ so it can be denoted as $\mathfrak{h}_{X}\left(
x_{1},x_{2},x_{3},...\right)  .$ Another way of expressing this is to say that
the matrix $\tilde{M}^{X}$ of size $\left\vert \mathsf{{Tot}}_{X}\right\vert
\times\left\vert \mathsf{{Tot}}_{X}\right\vert ,$ with entries $\left(
\tilde{M}^{X}\right)  _{PQ}=m_{PQ}\left(  x_{1},x_{2},x_{3},...\right)  $ is
\textit{stochastic}, i.e. the vector $\left(  1,1,...,1\right)  $ is the right
eigenvector, with the eigenvalue $\mathfrak{h}_{X}\left(  x_{1},x_{2}
,x_{3},...\right)  .$

\vskip.2cm
By replacing the monomials $m_{PQ}\left(  x_{1},x_{2},x_{3},...\right)  $ with
variables $a_{\varepsilon_{PQ}}$ one obtains from $\tilde{M}^{X}$ our matrix
$M^{X}.$

\begin{remark}
As we just said, we know from \cite{OS} that the rows of the matrix $M^{X}$
consist of the same matrix elements, permuted. So it is tempting to consider
the set of permutations $\pi_{PP^{\prime}}\in S_{\left\vert \mathsf{{Tot}}
_{X}\right\vert },$ which permute the elements of the row $P$ to these of row
$P^{\prime}.$ Unfortunately, rows of the matrix $M^{X}$ can contain repeated
elements, so the permutations $\pi_{PP^{\prime}}$ are not uniquely defined. \end{remark}

\subsection{Examples}

Here we present several examples in which our posets $X$ correspond to
partitions;  we first list the linear orders, that is, the standard Young tableaux of a given shape, and then present the
pedestal matrix with lines and columns labelled by the standard Young tableaux in the listed order.

\vskip.2cm \textbf{0.} In all examples we considered the pedestal matrix is diagonalisable in the
generic point. However for special values of variables the pedestal matrix might have non-trivial
Jordan blocks.  We give a minimal example - partition (3,1). It is essentially the same example as
the one before the main theorem, with the pedestal matrix (\ref{par31}), because the box (1,1)  comes first in any linear order and can be omitted.

\vskip .4cm
\noindent Here it is enough to take a partial evaluation $a_{10}\mapsto -2a_{01}$. Then the Jordan form is

\[ \left(\begin{array}{ccc}a_{00}-a_{01}&1&0\\0&a_{00}-a_{01}&0\\0&0&a_{00}+2a_{01}
\end{array}\right)\ .\]

\vskip .6cm
It would be interesting to understand the regimes in which the pedestal matrix is not diagonalisable.

\paragraph{1. Partition (3,2).}  The standard tableaux are

\vskip .8cm
\begin{tikzpicture}[scale=.6]
\draw (0,0) --(3,0);\draw (0,-1) --(3,-1);
\draw (0,0) --(0,-2);\draw (1,0) --(1,-2);\draw (2,0) --(2,-2);
\draw (0,-2) --(2,-2);\draw (3,0) --(3,-1);
\node at ( .5, -.5) {1};;\node at (1 .5, -.5) {3};\node at ( 2.5, -.5) {5};
\node at ( .5, -1.5) {2};\node at (1 .5, -1.5) {4};

\draw (0+4,0) --(3+4,0);\draw (+4,-1) --(3+4,-1);
\draw (+4,0) --(+4,-2);\draw (1+4,0) --(1+4,-2);\draw (2+4,0) --(2+4,-2);
\draw (+4,-2) --(2+4,-2);\draw (3+4,0) --(3+4,-1);
\node at ( .5+4, -.5) {1};;\node at (1 .5+4, -.5) {2};\node at ( 2.5+4, -.5) {5};
\node at ( .5+4, -1.5) {3};\node at (1 .5+4, -1.5) {4};

\draw (0+8,0) --(3+8,0);\draw (+8,-1) --(3+8,-1);
\draw (+8,0) --(+8,-2);\draw (1+8,0) --(1+8,-2);\draw (2+8,0) --(2+8,-2);
\draw (+8,-2) --(2+8,-2);\draw (3+8,0) --(3+8,-1);
\node at ( .5+8, -.5) {1};;\node at (1 .5+8, -.5) {3};\node at ( 2.5+8, -.5) {4};
\node at ( .5+8, -1.5) {2};\node at (1 .5+8, -1.5) {5};

\draw (0+12,0) --(3+12,0);\draw (+12,-1) --(3+12,-1);
\draw (+12,0) --(+12,-2);\draw (1+12,0) --(1+12,-2);\draw (2+12,0) --(2+12,-2);
\draw (+12,-2) --(2+12,-2);\draw (3+12,0) --(3+12,-1);
\node at ( .5+12, -.5) {1};;\node at (1 .5+12, -.5) {2};\node at ( 2.5+12, -.5) {4};
\node at ( .5+12, -1.5) {3};\node at (1 .5+12, -1.5) {5};

\draw (0+16,0) --(3+16,0);\draw (+16,-1) --(3+16,-1);
\draw (+16,0) --(+16,-2);\draw (1+16,0) --(1+16,-2);\draw (2+16,0) --(2+16,-2);
\draw (+16,-2) --(2+16,-2);\draw (3+16,0) --(3+16,-1);
\node at ( .5+16, -.5) {1};;\node at (1 .5+16, -.5) {2};\node at ( 2.5+16, -.5) {3};
\node at ( .5+16, -1.5) {4};\node at (1 .5+16, -1.5) {5};

\node at ( 3.44, -1.2) {,};\node at ( 3.44+4, -1.2) {,};\node at ( 3.44+8, -1.2) {,};\node at ( 3.44+12, -1.2) {,};
\node at ( 3.44+16, -1.2) {.};
\end{tikzpicture}

\vskip .6cm
The pedestal matrix $\tilde{M}^{X}$ is $x_{1}^{2}A_{(3,2)},$ where

\[
A_{(3,2)}=\left(
\begin{array}
[c]{ccccc}
x_{1}^{3} & x_{2}^{3} & x_{1}^{2}x_{2} & x_{2}^{2}x_{3} & x_{1}x_{2}^{2}\\
x_{2}^{3} & x_{1}^{3} & x_{2}^{2}x_{3} & x_{1}^{2}x_{2} & x_{1}x_{2}^{2}\\
x_{1}^{2}x_{2} & x_{2}^{2}x_{3} & x_{1}^{3} & x_{2}^{3} & x_{1}x_{2}^{2}\\
x_{2}^{2}x_{3} & x_{1}^{2}x_{2} & x_{2}^{3} & x_{1}^{3} & x_{1}x_{2}^{2}\\
x_{2}^{2}x_{3} & x_{1}^{2}x_{2} & x_{2}^{3} & x_{1}x_{2}^{2} & x_{1}^{3}
\end{array}
\right)  \ .
\]

\vskip .1cm
\noindent  After a replacement
\begin{equation}
\phi:(x_{1}^{3},x_{1}^{2}x_{2},x_{1}x_{2}^{2},x_{2}^{3},x_{2}^{2}
x_{3})\rightarrow(a_{1},a_{2},a_{3},a_{4},a_{5})\ , \label{repl32}
\end{equation}

\vskip .1cm\noindent
we have

\[
A_{(3,2)}^{\phi}=\left(
\begin{array}
[c]{ccccc}
a_{1} & a_{4} & a_{2} & a_{5} & a_{3}\\
a_{4} & a_{1} & a_{5} & a_{2} & a_{3}\\
a_{2} & a_{5} & a_{1} & a_{4} & a_{3}\\
a_{5} & a_{2} & a_{4} & a_{1} & a_{3}\\
a_{5} & a_{2} & a_{4} & a_{3} & a_{1}
\end{array}
\right)  \ .
\]
The eigenvalues of $A_{(3,2)}^{\phi}$ are
\[
a_{1}-a_{3}\ ,\ a_{1}+a_{2}-a_{4}-a_{5}\ ,\ a_{1}-a_{2}+a_{4}-a_{5}
\ ,\ a_{1}-a_{2}-a_{4}+a_{5}\ ,\ a_{1}+a_{2}+a_{3}+a_{4}+a_{5}\ .
\]

\paragraph{2. Partition (3,1,1).} The standard tableaux are

\vskip .8cm
\begin{tikzpicture}[scale=.6]
\draw (0,0) --(3,0);\draw (0,-1) --(3,-1);\draw (0,-2) --(1,-2);\draw (0,-3) --(1,-3);
\draw (0,0) --(0,-3);\draw (1,0) --(1,-3);\draw (2,0) --(2,-1);
\draw (3,0) --(3,-1);
\node at ( .5, -.5) {1};\node at (1 .5, -.5) {4};\node at ( 2.5, -.5) {5};
\node at ( .5, -1.5) {2};\node at (0 .5, -2.5) {3};

\draw (0+4,0) --(3+4,0);\draw (0+4,-1) --(3+4,-1);\draw (0+4,-2) --(1+4,-2);\draw (0+4,-3) --(1+4,-3);
\draw (0+4,0) --(0+4,-3);\draw (1+4,0) --(1+4,-3);\draw (2+4,0) --(2+4,-1);
\draw (3+4,0) --(3+4,-1);
\node at ( .5+4, -.5) {1};;\node at (1 .5+4, -.5) {3};\node at ( 2.5+4, -.5) {5};
\node at ( .5+4, -1.5) {2};\node at (0 .5+4, -2.5) {4};

\draw (0+8,0) --(3+8,0);\draw (0+8,-1) --(3+8,-1);\draw (0+8,-2) --(1+8,-2);\draw (0+8,-3) --(1+8,-3);
\draw (0+8,0) --(0+8,-3);\draw (1+8,0) --(1+8,-3);\draw (2+8,0) --(2+8,-1);
\draw (3+8,0) --(3+8,-1);
\node at ( .5+8, -.5) {1};;\node at (1 .5+8, -.5) {2};\node at ( 2.5+8, -.5) {5};
\node at ( .5+8, -1.5) {3};\node at (0 .5+8, -2.5) {4};

\draw (0+12,0) --(3+12,0);\draw (0+12,-1) --(3+12,-1);\draw (0+12,-2) --(1+12,-2);\draw (0+12,-3) --(1+12,-3);
\draw (0+12,0) --(0+12,-3);\draw (1+12,0) --(1+12,-3);\draw (2+12,0) --(2+12,-1);
\draw (3+12,0) --(3+12,-1);
\node at ( .5+12, -.5) {1};;\node at (1 .5+12, -.5) {3};\node at ( 2.5+12, -.5) {4};
\node at ( .5+12, -1.5) {2};\node at (0 .5+12, -2.5) {5};

\node at ( 3.4, -1.2) {,};\node at ( 3.4+4, -1.2) {,};\node at ( 3.4+8, -1.2) {,};\node at ( 3.4+12, -1.2) {,};
\end{tikzpicture}

\vskip .8cm
\begin{tikzpicture}[scale=.6]

\draw (0+16,0) --(3+16,0);\draw (0+16,-1) --(3+16,-1);\draw (0+16,-2) --(1+16,-2);\draw (0+16,-3) --(1+16,-3);
\draw (0+16,0) --(0+16,-3);\draw (1+16,0) --(1+16,-3);\draw (2+16,0) --(2+16,-1);
\draw (3+16,0) --(3+16,-1);
\node at ( .5+16, -.5) {1};;\node at (1 .5+16, -.5) {2};\node at ( 2.5+16, -.5) {4};
\node at ( .5+16, -1.5) {3};\node at (0 .5+16, -2.5) {5};

\draw (0+20,0) --(3+20,0);\draw (0+20,-1) --(3+20,-1);\draw (0+20,-2) --(1+20,-2);\draw (0+20,-3) --(1+20,-3);
\draw (0+20,0) --(0+20,-3);\draw (1+20,0) --(1+20,-3);\draw (2+20,0) --(2+20,-1);
\draw (3+20,0) --(3+20,-1);
\node at ( .5+20, -.5) {1};;\node at (1 .5+20, -.5) {2};\node at ( 2.5+20, -.5) {3};
\node at ( .5+20, -1.5) {4};\node at (0 .5+20, -2.5) {5};

\node at ( 3.4+16, -1.2) {,};\node at ( 3.4+20, -1.2) {.};
\end{tikzpicture}

\vskip .6cm
The pedestal matrix is $x_{1}^{2}A_{(3,1,1)}$ where

\[
A_{(3,1,1)}=\left(
\begin{array}
[c]{cccccc}
x_{1}^{3} & x_{1}x_{2}^{2} & x_{2}^{3} & x_{1}^{2}x_{2} & x_{2}^{2}x_{3} &
x_{1}x_{2}^{2}\\
x_{1}x_{2}^{2} & x_{1}^{3} & x_{2}^{3} & x_{1}^{2}x_{2} & x_{2}^{2}x_{3} &
x_{1}x_{2}^{2}\\
x_{1}x_{2}^{2} & x_{2}^{3} & x_{1}^{3} & x_{2}^{2}x_{3} & x_{1}^{2}x_{2}^{2} &
x_{1}x_{2}^{2}\\
x_{1}x_{2}^{2} & x_{1}^{2}x_{2} & x_{2}^{2}x_{3} & x_{1}^{3} & x_{2}^{3} &
x_{1}x_{2}^{2}\\
x_{1}x_{2}^{2} & x_{2}^{2}x_{3} & x_{1}^{2}x_{2} & x_{2}^{3} & x_{1}^{3} &
x_{1}x_{2}^{2}\\
x_{1}x_{2}^{2} & x_{2}^{2}x_{3} & x_{1}^{2}x_{2} & x_{2}^{3} & x_{1}x_{2}^{2}
& x_{1}^{3}
\end{array}
\right)  \ .
\]

\vskip .1cm\noindent After the same replacement (\ref{repl32}) (the matrix $A_{(3,1,1)}$ contains
the same monomials as the matrix $A_{(3,2)}$) we have

\[
A_{(3,1,1)}^{\phi}=\left(
\begin{array}
[c]{cccccc}
a_{1} & a_{3} & a_{4} & a_{2} & a_{5} & a_{3}\\
a_{3} & a_{1} & a_{4} & a_{2} & a_{5} & a_{3}\\
a_{3} & a_{4} & a_{1} & a_{5} & a_{2} & a_{3}\\
a_{3} & a_{2} & a_{5} & a_{1} & a_{4} & a_{3}\\
a_{3} & a_{5} & a_{2} & a_{4} & a_{1} & a_{3}\\
a_{3} & a_{5} & a_{2} & a_{4} & a_{3} & a_{1}
\end{array}
\right)  \ .
\]

\vskip .1cm\noindent  The eigenvalues of $A_{(3,1,1)}^{\phi}$ are (the notation $(y)_{k}$ means that
the multiplicity of the eigenvalue $y$ is $k$)
\[
(a_{1}-a_{3})_{2}\ ,\ a_{1}+a_{2}-a_{4}-a_{5}\ ,\ a_{1}-a_{2}+a_{4}
-a_{5}\ ,\ a_{1}-a_{2}-a_{4}+a_{5}\ ,\ a_{1}+a_{2}+2a_{3}+a_{4}+a_{5}\ .
\]

The example $(3,1,1)$ shows degeneration: the letter $a_{3}$ appears twice in
every row of $A_{(3,1,1)}^{\phi}$. The corresponding monomial is $x_{1}
^{3}x_{2}^{2}$ so for writing down the decomposition of the matrix $B_{a_{3}}$
we need filters from $\mathcal{F}_{3,2}$. There are three of them in
$\mathcal{F}_{3,2}$ (the notation is like for a matrix; element $(i,j)$ is in
the intersection of row $i$ and column $j$):

\begin{itemize}
\item[$\bullet$] $F_{1}$: Floor 1 contains cells (1,1), (1,2) and (2,1);

\item[$\bullet$] $F_{2}$: Floor 1 contains cells (1,1), (1,2) and (1,3);

\item[$\bullet$] $F_{3}$: Floor 1 contains cells (1,1), (2,1) and (3,1).
\end{itemize}

The matrices of action of these filters on the linear orders are

\[
M_{F_{1}}=\left(
\begin{array}
[c]{cccccc}
0 & 1 & 0 & 0 & 0 & 0\\
0 & 1 & 0 & 0 & 0 & 0\\
0 & 0 & 1 & 0 & 0 & 0\\
0 & 0 & 0 & 1 & 0 & 0\\
0 & 0 & 0 & 0 & 1 & 0\\
0 & 0 & 0 & 0 & 1 & 0
\end{array}
\right)  ,M_{F_{2}}=\left(
\begin{array}
[c]{cccccc}
0 & 0 & 0 & 0 & 0 & 1\\
0 & 0 & 0 & 0 & 0 & 1\\
0 & 0 & 0 & 0 & 0 & 1\\
0 & 0 & 0 & 0 & 0 & 1\\
0 & 0 & 0 & 0 & 0 & 1\\
0 & 0 & 0 & 0 & 0 & 1
\end{array}
\right)  ,
\]

\vskip .2cm
\[
M_{F_{3}}=\left(
\begin{array}
[c]{cccccc}
1 & 0 & 0 & 0 & 0 & 0\\
1 & 0 & 0 & 0 & 0 & 0\\
1 & 0 & 0 & 0 & 0 & 0\\
1 & 0 & 0 & 0 & 0 & 0\\
1 & 0 & 0 & 0 & 0 & 0\\
1 & 0 & 0 & 0 & 0 & 0
\end{array}
\right)  .
\]

\vskip .2cm\noindent  The family $\mathcal{F}_{5}$ contains one filter, which acts as the identity
$I$. The matrix $B_{a_{3}}$ is thus

\[
B_{a_{3}}=\left(
\begin{array}
[c]{cccccc}
0 & 1 & 0 & 0 & 0 & 1\\
1 & 0 & 0 & 0 & 0 & 1\\
1 & 0 & 0 & 0 & 0 & 1\\
1 & 0 & 0 & 0 & 0 & 1\\
1 & 0 & 0 & 0 & 0 & 1\\
1 & 0 & 0 & 0 & 1 & 0
\end{array}
\right)  =M_{F_{1}}+M_{F_{2}}+M_{F_{3}}-I\ ,
\]

\vskip .2cm\noindent as dictated by the inclusion-exclusion formula.

\paragraph{3. Partition (3,2,1).} In this example, to save the space, we write down the pedestal matrix in which the replacement
\[
(x_{1}^{6},x_{1}^{5}x_{2},x_{1}^{4}x_{2}^{2},x_{1}^{4}x_{2}x_{3},x_{1}
^{3}x_{2}^{3},x_{1}^{3}x_{2}^{2}x_{3},x_{1}^{2}x_{2}^{4},x_{1}^{2}x_{2}
^{3}x_{3},x_{1}^{2}x_{2}^{2}x_{3}^{2},x_{1}^{2}x_{2}^{2}x_{3}x_{4})\rightarrow
\]
\[
(a_{1},a_{2},a_{3},a_{4},a_{5},a_{6},a_{7},a_{8},a_{9},a_{10})
\]
is already made.

The standard tableaux are

\vskip .6cm
\begin{tikzpicture}[scale=.6]
\draw (0,0) --(3,0);\draw (0,-1) --(3,-1);\draw (0,-2) --(2,-2);\draw (0,-3) --(1,-3);
\draw (0,0) --(0,-3);\draw (1,0) --(1,-3);\draw (2,0) --(2,-2);
\draw (3,0) --(3,-1);
\node at ( .5, -.5) {1};\node at (1 .5, -.5) {4};\node at ( 2.5, -.5) {6};
\node at ( .5, -1.5) {2};\node at (1 .5, -1.5) {5};\node at (0 .5, -2.5) {3};

\draw (0+5,0) --(3+5,0);\draw (0+5,-1) --(3+5,-1);\draw (0+5,-2) --(2+5,-2);\draw (0+5,-3) --(1+5,-3);
\draw (0+5,0) --(0+5,-3);\draw (1+5,0) --(1+5,-3);\draw (2+5,0) --(2+5,-2);
\draw (3+5,0) --(3+5,-1);
\node at ( .5+5, -.5) {1};\node at (1 .5+5, -.5) {3};\node at ( 2.5+5, -.5) {6};
\node at ( .5+5, -1.5) {2};\node at (1 .5+5, -1.5) {5};\node at (0 .5+5, -2.5) {4};

\draw (0+10,0) --(3+10,0);\draw (0+10,-1) --(3+10,-1);\draw (0+10,-2) --(2+10,-2);\draw (0+10,-3) --(1+10,-3);
\draw (0+10,0) --(0+10,-3);\draw (1+10,0) --(1+10,-3);\draw (2+10,0) --(2+10,-2);
\draw (3+10,0) --(3+10,-1);
\node at ( .5+10, -.5) {1};\node at (1 .5+10, -.5) {2};\node at ( 2.5+10, -.5) {6};
\node at ( .5+10, -1.5) {3};\node at (1 .5+10, -1.5) {5};\node at (0 .5+10, -2.5) {4};

\draw (0+15,0) --(3+15,0);\draw (0+15,-1) --(3+15,-1);\draw (0+15,-2) --(2+15,-2);\draw (0+15,-3) --(1+15,-3);
\draw (0+15,0) --(0+15,-3);\draw (1+15,0) --(1+15,-3);\draw (2+15,0) --(2+15,-2);
\draw (3+15,0) --(3+15,-1);
\node at ( .5+15, -.5) {1};\node at (1 .5+15, -.5) {3};\node at ( 2.5+15, -.5) {6};
\node at ( .5+15, -1.5) {2};\node at (1 .5+15, -1.5) {4};\node at (0 .5+15, -2.5) {5};

\node at ( 3.9, -1.2) {,};\node at ( 3.4+5.5, -1.2) {,};\node at ( 3.4+10.5, -1.2) {,};\node at ( 3.4+15.5, -1.2) {,};

\end{tikzpicture}

\vskip .5cm
\begin{tikzpicture}[scale=.6]
\draw (0,0) --(3,0);\draw (0,-1) --(3,-1);\draw (0,-2) --(2,-2);\draw (0,-3) --(1,-3);
\draw (0,0) --(0,-3);\draw (1,0) --(1,-3);\draw (2,0) --(2,-2);
\draw (3,0) --(3,-1);
\node at ( .5, -.5) {1};\node at (1 .5, -.5) {2};\node at ( 2.5, -.5) {6};
\node at ( .5, -1.5) {3};\node at (1 .5, -1.5) {4};\node at (0 .5, -2.5) {5};

\draw (0+5,0) --(3+5,0);\draw (0+5,-1) --(3+5,-1);\draw (0+5,-2) --(2+5,-2);\draw (0+5,-3) --(1+5,-3);
\draw (0+5,0) --(0+5,-3);\draw (1+5,0) --(1+5,-3);\draw (2+5,0) --(2+5,-2);
\draw (3+5,0) --(3+5,-1);
\node at ( .5+5, -.5) {1};\node at (1 .5+5, -.5) {4};\node at ( 2.5+5, -.5) {5};
\node at ( .5+5, -1.5) {2};\node at (1 .5+5, -1.5) {6};\node at (0 .5+5, -2.5) {3};

\draw (0+10,0) --(3+10,0);\draw (0+10,-1) --(3+10,-1);\draw (0+10,-2) --(2+10,-2);\draw (0+10,-3) --(1+10,-3);
\draw (0+10,0) --(0+10,-3);\draw (1+10,0) --(1+10,-3);\draw (2+10,0) --(2+10,-2);
\draw (3+10,0) --(3+10,-1);
\node at ( .5+10, -.5) {1};\node at (1 .5+10, -.5) {3};\node at ( 2.5+10, -.5) {5};
\node at ( .5+10, -1.5) {2};\node at (1 .5+10, -1.5) {6};\node at (0 .5+10, -2.5) {4};

\draw (0+15,0) --(3+15,0);\draw (0+15,-1) --(3+15,-1);\draw (0+15,-2) --(2+15,-2);\draw (0+15,-3) --(1+15,-3);
\draw (0+15,0) --(0+15,-3);\draw (1+15,0) --(1+15,-3);\draw (2+15,0) --(2+15,-2);
\draw (3+15,0) --(3+15,-1);
\node at ( .5+15, -.5) {1};\node at (1 .5+15, -.5) {2};\node at ( 2.5+15, -.5) {5};
\node at ( .5+15, -1.5) {3};\node at (1 .5+15, -1.5) {6};\node at (0 .5+15, -2.5) {4};

\node at ( 3.9, -1.2) {,};\node at ( 3.4+5.5, -1.2) {,};\node at ( 3.4+10.5, -1.2) {,};\node at ( 3.4+15.5, -1.2) {,};

\end{tikzpicture}

\vskip .5cm
\begin{tikzpicture}[scale=.6]
\draw (0,0) --(3,0);\draw (0,-1) --(3,-1);\draw (0,-2) --(2,-2);\draw (0,-3) --(1,-3);
\draw (0,0) --(0,-3);\draw (1,0) --(1,-3);\draw (2,0) --(2,-2);
\draw (3,0) --(3,-1);
\node at ( .5, -.5) {1};\node at (1 .5, -.5) {3};\node at ( 2.5, -.5) {4};
\node at ( .5, -1.5) {2};\node at (1 .5, -1.5) {6};\node at (0 .5, -2.5) {5};

\draw (0+5,0) --(3+5,0);\draw (0+5,-1) --(3+5,-1);\draw (0+5,-2) --(2+5,-2);\draw (0+5,-3) --(1+5,-3);
\draw (0+5,0) --(0+5,-3);\draw (1+5,0) --(1+5,-3);\draw (2+5,0) --(2+5,-2);
\draw (3+5,0) --(3+5,-1);
\node at ( .5+5, -.5) {1};\node at (1 .5+5, -.5) {2};\node at ( 2.5+5, -.5) {4};
\node at ( .5+5, -1.5) {3};\node at (1 .5+5, -1.5) {6};\node at (0 .5+5, -2.5) {5};

\draw (0+10,0) --(3+10,0);\draw (0+10,-1) --(3+10,-1);\draw (0+10,-2) --(2+10,-2);\draw (0+10,-3) --(1+10,-3);
\draw (0+10,0) --(0+10,-3);\draw (1+10,0) --(1+10,-3);\draw (2+10,0) --(2+10,-2);
\draw (3+10,0) --(3+10,-1);
\node at ( .5+10, -.5) {1};\node at (1 .5+10, -.5) {2};\node at ( 2.5+10, -.5) {3};
\node at ( .5+10, -1.5) {4};\node at (1 .5+10, -1.5) {6};\node at (0 .5+10, -2.5) {5};

\draw (0+15,0) --(3+15,0);\draw (0+15,-1) --(3+15,-1);\draw (0+15,-2) --(2+15,-2);\draw (0+15,-3) --(1+15,-3);
\draw (0+15,0) --(0+15,-3);\draw (1+15,0) --(1+15,-3);\draw (2+15,0) --(2+15,-2);
\draw (3+15,0) --(3+15,-1);
\node at ( .5+15, -.5) {1};\node at (1 .5+15, -.5) {3};\node at ( 2.5+15, -.5) {5};
\node at ( .5+15, -1.5) {2};\node at (1 .5+15, -1.5) {4};\node at (0 .5+15, -2.5) {6};

\node at ( 3.9, -1.2) {,};\node at ( 3.4+5.5, -1.2) {,};\node at ( 3.4+10.5, -1.2) {,};\node at ( 3.4+15.5, -1.2) {,};

\end{tikzpicture}

\vskip .5cm
\begin{tikzpicture}[scale=.6]
\draw (0,0) --(3,0);\draw (0,-1) --(3,-1);\draw (0,-2) --(2,-2);\draw (0,-3) --(1,-3);
\draw (0,0) --(0,-3);\draw (1,0) --(1,-3);\draw (2,0) --(2,-2);
\draw (3,0) --(3,-1);
\node at ( .5, -.5) {1};\node at (1 .5, -.5) {2};\node at ( 2.5, -.5) {5};
\node at ( .5, -1.5) {3};\node at (1 .5, -1.5) {4};\node at (0 .5, -2.5) {6};

\draw (0+5,0) --(3+5,0);\draw (0+5,-1) --(3+5,-1);\draw (0+5,-2) --(2+5,-2);\draw (0+5,-3) --(1+5,-3);
\draw (0+5,0) --(0+5,-3);\draw (1+5,0) --(1+5,-3);\draw (2+5,0) --(2+5,-2);
\draw (3+5,0) --(3+5,-1);
\node at ( .5+5, -.5) {1};\node at (1 .5+5, -.5) {3};\node at ( 2.5+5, -.5) {4};
\node at ( .5+5, -1.5) {2};\node at (1 .5+5, -1.5) {5};\node at (0 .5+5, -2.5) {6};

\draw (0+10,0) --(3+10,0);\draw (0+10,-1) --(3+10,-1);\draw (0+10,-2) --(2+10,-2);\draw (0+10,-3) --(1+10,-3);
\draw (0+10,0) --(0+10,-3);\draw (1+10,0) --(1+10,-3);\draw (2+10,0) --(2+10,-2);
\draw (3+10,0) --(3+10,-1);
\node at ( .5+10, -.5) {1};\node at (1 .5+10, -.5) {2};\node at ( 2.5+10, -.5) {4};
\node at ( .5+10, -1.5) {3};\node at (1 .5+10, -1.5) {5};\node at (0 .5+10, -2.5) {6};

\draw (0+15,0) --(3+15,0);\draw (0+15,-1) --(3+15,-1);\draw (0+15,-2) --(2+15,-2);\draw (0+15,-3) --(1+15,-3);
\draw (0+15,0) --(0+15,-3);\draw (1+15,0) --(1+15,-3);\draw (2+15,0) --(2+15,-2);
\draw (3+15,0) --(3+15,-1);
\node at ( .5+15, -.5) {1};\node at (1 .5+15, -.5) {2};\node at ( 2.5+15, -.5) {3};
\node at ( .5+15, -1.5) {4};\node at (1 .5+15, -1.5) {5};\node at (0 .5+15, -2.5) {6};

\node at ( 3.9, -1.2) {,};\node at ( 3.4+5.5, -1.2) {,};\node at ( 3.4+10.5, -1.2) {,};\node at ( 3.4+15.5, -1.2) {.};

\end{tikzpicture}

\vskip .7cm
The matrix $A_{(3,2,1)}^{\phi}$ is
\[
A_{(3,2,1)}^{\phi}=\left(
\begin{array}
[c]{cccccccccccccccc}
\!a_{1}\! & \!a_{5}\! & \!a_{7}\! & \!a_{3}\! & \!a_{9}\! & \!a_{2}\! &
\!a_{6}\! & \!a_{8}\! & \!a_{3}\! & \!a_{9}\! & \!a_{5}\! & \!a_{2}\! &
\!a_{8}\! & \!a_{4}\! & \!a_{10}\! & \!a_{6}\!\\
\!a_{5}\! & \!a_{1}\! & \!a_{7}\! & \!a_{3}\! & \!a_{9}\! & \!a_{6}\! &
\!a_{2}\! & \!a_{8}\! & \!a_{3}\! & \!a_{9}\! & \!a_{5}\! & \!a_{2}\! &
\!a_{8}\! & \!a_{4}\! & \!a_{10}\! & \!a_{6}\!\\
\!a_{5}\! & \!a_{7}\! & \!a_{1}\! & \!a_{9}\! & \!a_{3}\! & \!a_{6}\! &
\!a_{8}\! & \!a_{2}\! & \!a_{9}\! & \!a_{3}\! & \!a_{5}\! & \!a_{8}\! &
\!a_{2}\! & \!a_{10}\! & \!a_{4}\! & \!a_{6}\!\\
\!a_{5}\! & \!a_{3}\! & \!a_{9}\! & \!a_{1}\! & \!a_{7}\! & \!a_{6}\! &
\!a_{2}\! & \!a_{8}\! & \!a_{4}\! & \!a_{10}\! & \!a_{6}\! & \!a_{2}\! &
\!a_{8}\! & \!a_{3}\! & \!a_{9}\! & \!a_{5}\!\\
\!a_{5}\! & \!a_{9}\! & \!a_{3}\! & \!a_{7}\! & \!a_{1}\! & \!a_{6}\! &
\!a_{8}\! & \!a_{2}\! & \!a_{10}\! & \!a_{4}\! & \!a_{6}\! & \!a_{8}\! &
\!a_{2}\! & \!a_{9}\! & \!a_{3}\! & \!a_{5}\!\\
\!a_{2}\! & \!a_{6}\! & \!a_{8}\! & \!a_{3}\! & \!a_{9}\! & \!a_{1}\! &
\!a_{5}\! & \!a_{7}\! & \!a_{3}\! & \!a_{9}\! & \!a_{5}\! & \!a_{4}\! &
\!a_{10}\! & \!a_{2}\! & \!a_{8}\! & \!a_{6}\!\\
\!a_{6}\! & \!a_{2}\! & \!a_{8}\! & \!a_{3}\! & \!a_{9}\! & \!a_{5}\! &
\!a_{1}\! & \!a_{7}\! & \!a_{3}\! & \!a_{9}\! & \!a_{5}\! & \!a_{4}\! &
\!a_{10}\! & \!a_{2}\! & \!a_{8}\! & \!a_{6}\!\\
\!a_{6}\! & \!a_{8}\! & \!a_{2}\! & \!a_{9}\! & \!a_{3}\! & \!a_{5}\! &
\!a_{7}\! & \!a_{1}\! & \!a_{9}\! & \!a_{3}\! & \!a_{5}\! & \!a_{10}\! &
\!a_{4}\! & \!a_{8}\! & \!a_{2}\! & \!a_{6}\!\\
\!a_{6}\! & \!a_{2}\! & \!a_{8}\! & \!a_{4}\! & \!a_{10}\! & \!a_{5}\! &
\!a_{3}\! & \!a_{9}\! & \!a_{1}\! & \!a_{7}\! & \!a_{5}\! & \!a_{3}\! &
\!a_{9}\! & \!a_{2}\! & \!a_{8}\! & \!a_{6}\!\\
\!a_{6}\! & \!a_{8}\! & \!a_{2}\! & \!a_{10}\! & \!a_{4}\! & \!a_{5}\! &
\!a_{9}\! & \!a_{3}\! & \!a_{7}\! & \!a_{1}\! & \!a_{5}\! & \!a_{9}\! &
\!a_{3}\! & \!a_{8}\! & \!a_{2}\! & \!a_{6}\!\\
\!a_{6}\! & \!a_{8}\! & \!a_{2}\! & \!a_{10}\! & \!a_{4}\! & \!a_{5}\! &
\!a_{9}\! & \!a_{3}\! & \!a_{7}\! & \!a_{5}\! & \!a_{1}\! & \!a_{9}\! &
\!a_{3}\! & \!a_{8}\! & \!a_{6}\! & \!a_{2}\!\\
\!a_{5}\! & \!a_{3}\! & \!a_{9}\! & \!a_{2}\! & \!a_{8}\! & \!a_{6}\! &
\!a_{4}\! & \!a_{10}\! & \!a_{2}\! & \!a_{8}\! & \!a_{6}\! & \!a_{1}\! &
\!a_{7}\! & \!a_{3}\! & \!a_{9}\! & \!a_{5}\!\\
\!a_{5}\! & \!a_{9}\! & \!a_{3}\! & \!a_{8}\! & \!a_{2}\! & \!a_{6}\! &
\!a_{10}\! & \!a_{4}\! & \!a_{8}\! & \!a_{2}\! & \!a_{6}\! & \!a_{7}\! &
\!a_{1}\! & \!a_{9}\! & \!a_{3}\! & \!a_{5}\!\\
\!a_{6}\! & \!a_{4}\! & \!a_{10}\! & \!a_{2}\! & \!a_{8}\! & \!a_{5}\! &
\!a_{3}\! & \!a_{9}\! & \!a_{2}\! & \!a_{8}\! & \!a_{6}\! & \!a_{3}\! &
\!a_{9}\! & \!a_{1}\! & \!a_{7}\! & \!a_{5}\!\\
\!a_{6}\! & \!a_{10}\! & \!a_{4}\! & \!a_{8}\! & \!a_{2}\! & \!a_{5}\! &
\!a_{9}\! & \!a_{3}\! & \!a_{9}\! & \!a_{2}\! & \!a_{6}\! & \!a_{9}\! &
\!a_{3}\! & \!a_{7}\! & \!a_{1}\! & \!a_{5}\!\\
\!a_{6}\! & \!a_{10}\! & \!a_{4}\! & \!a_{8}\! & \!a_{2}\! & \!a_{5}\! &
\!a_{9}\! & \!a_{3}\! & \!a_{8}\! & \!a_{6}\! & \!a_{2}\! & \!a_{9}\! &
\!a_{3}\! & \!a_{7}\! & \!a_{5}\! & \!a_{1}\!
\end{array}
\right)  \ .
\]
The eigenvalues of $A_{(3,2,1)}^{\phi}$ are
\[
(a_{1}-a_{4}-a_{7}+a_{10})_{3}\ ,\ a_{1}-a_{4}+a_{7}-a_{10}\ ,\ (a_{1}
+a_{2}-a_{5}-a_{6})_{2}\ ,\ (a_{1}-a_{2}-a_{5}+a_{6})_{2}\ ,
\]
\[
(a_{1}-a_{2}-a_{3}+a_{4}+a_{7}-a_{8}-a_{9}+a_{10})_{2}\ ,\ (a_{1}-a_{2}
-a_{3}+a_{4}-a_{7}+a_{8}+a_{9}-a_{10})_{2}\ ,
\]
\[
(a_{1}-a_{4}+a_{5}-a_{6}+a_{7}-a_{10})_{2}\ ,\ a_{1}+2a_{2}+2a_{3}+a_{4}
-a_{7}-2a_{8}-2a_{9}-a_{10}\ ,
\]
\[
a_{1}+2a_{2}+2a_{3}+2a_{5}+2a_{6}+a_{7}+2a_{8}+2a_{9}+a_{10}\ .
\]

\vskip .4cm
\textbf{Acknowledgements.} Part of this work was done in IHES during the program dedicated to the
100th anniversary of the Ising model, in May - June of 2022. R.K., M.K. and S.S. are grateful to the organizers
of the program for the uplifting atmosphere.
R.K. was supported by NSF grant DMS-1940932 and the Simons Foundation grant 327929.
The work of W.S. was supported by the NSF grant DMS-2101491
and by the Sloan Research Fellowship. The work of S.S. was supported by the RSF under project 23-11-00150.

We thank Professor P. Diaconis and Professor R. Stanley for valuable remarks.

\end{document}